\documentclass[12pt,reqno]{amsart}
\usepackage{CJK}
\usepackage[top=2.5cm,bottom=2cm,left=2.5cm,right=2.5cm]{geometry}
\usepackage{amsmath}
\usepackage{exscale}
\usepackage{relsize}
\usepackage{bm}
\usepackage{latexsym}
\numberwithin{equation}{section}

\newtheorem{theorem}{Theorem}[section]

\newtheorem{lemma}[theorem]{Lemma}
\newtheorem{remark}{Remark}[section]

\begin{document}

\title[low mach limit of full compressible Navier-Stokes equations]{low mach number limit of full compressible Navier-Stokes equations with revised Maxwell law}

\author{Zhao Wang and Yuxi Hu}

\begin{abstract}
In this paper, we study the low Mach number limit of the full compressible Navier-Stokes equations with revised Maxwell law.  By applying the uniform estimation of the error system, we prove that the solutions of the full compressible Navier-Stokes equations with time relaxation converge to that of the incompressible Navier-Stokes equations as the Mach number tends to zero. Moreover, the convergence rates are also obtained.
\\[2em]
{\bf Keywords:}  Mach limit; Navier-Stokes equations; revised Maxwell law; energy methods
 \thanks{\noindent Zhao Wang , Department of Mathematics, China University of Mining and Technology, Beijing, 100083, P.R. China\\
\indent Yuxi Hu, Department of Mathematics, China University of Mining and Technology, Beijing, 100083, P.R. China, yxhu86@163.com
}
\end{abstract}
\maketitle
\section{\textbf{Introduction}}
In this paper, we study the low Mach number limit of the full Navier-Stokes equations with revised Maxwell law. The 3-d full compressible Navier-Stokes equations are constructed from continuity equation, momentum equation and energy equation which can be written in the following form (Euler coordinates)
\begin{equation}\label{c N-S}
  \begin{cases}
    \partial_t\rho+\text{div}(\rho u)=0,\\
    \partial_t(\rho u)+\text{div}(\rho u\otimes u)+\nabla p=\text{div}S,\\
    \partial_t[\rho(e+\frac{1}{2}u^2)]+\text{div}[\rho u(e+\frac{1}{2}u^2)+up]+\text{div}q=\text{div}(uS).
  \end{cases}
\end{equation}
Here $\rho$ and $u=( u_1,u_2,u_3)$ represent fluid density, velocity, respectively. The pressure $p=p(\rho,\theta)$ and  internal energy $e=(\rho,\theta)$ are smooth functions of the density $\rho$ and the absolute temperature $\theta$, satisfying
$$\rho^2 e_\rho(\rho,\theta)=p(\rho,\theta)-\theta p_\theta(\rho,\theta).$$
 For classical fluids, the stress tensor $S$ and heat flux $q$ are given by Newtonian law:
\begin{equation}\label{Newtonian law}
  S=\mu\left(\nabla u+(\nabla u)^\top -\frac{2}{3}\text{div}uI_3\right)+\lambda \text{div}uI_3,
\end{equation}
and Fourier's law:
\begin{equation}\label{Fourier}
  q+\kappa\nabla\theta=0.
\end{equation}
Where $I_3$ denotes the identity matrix in $\mathbb{R}^3$. $\nabla$ is the gradient operator with respect to the space variable $x=(x_1,x_2,x_3)$ and $\text{A}^\top$ stands for the transpose of the matrix $\text{A}$.  $\mu$, $\lambda$ and $\kappa$ are positive constants.

As is well-known, the Fourier's law \eqref{Fourier} predicts that heat waves have an infinite propagation speed. In order to overcome the contradiction, Cattaneo introduced the following well-known Cattaneo's heat transfer law:
\begin{equation}\label{Cattaneo}
 \tau\partial_tq+q+\kappa\nabla\theta=0.
\end{equation}

For macromolecular or polymeric fluids, Maxwell \cite{J. C. Maxwell} (1867) proposed so called Maxwell law :
\begin{equation}\label{Maxwell}
  \tau'\partial_tS+S=\mu\left(\nabla u+(\nabla u)^\top -\frac{2}{3}\text{div}uI_3\right)+\lambda \text{div}uI_3,
\end{equation}
with the relaxation parameter $\tau'>0$. This relation indicates there exists a time lag in response of stress tensor to velocity gradient and divergence. Actually, this time lag exists even in simple fluids, water for example, which can be negligible because of the small effect under macroscopic conditions. However, the time lag can not be negligible for other fluids or under other conditions, see \cite{MMA,PC,SRK} for more details. As for the model systems with equation \eqref{Maxwell}, see \cite{R. Racke and J. Saal} for incompressible Navier-Stokes equations.
Yong \cite{W.-A. Yong 2014} divided $S$ into two parts $S=S_1+S_2I_3$ and did the relaxation for each one as follows:
\begin{equation}\label{revised Maxwell}
  \begin{cases}
   \tau_1\partial_tS_1+S_1=\mu\left(\nabla u+(\nabla u)^\top -\frac{2}{3}\text{div}uI_3\right),\\
   \tau_2\partial_tS_2+S_2=\lambda \text{div}u.
  \end{cases}
\end{equation}
Here $S_1$ is a $n \times n$ square matrix and symmetric and traceless if it was symmetric and traceless initially, and $S_2$ is a scalar variable. From mathematic point of view, revised Maxwell's relation make the relaxed system (isentropic regime) symmetric. And as physical explanation, $\tau_1$ and $\tau_2$ are shear relaxation time and compressional relaxation time, see a similar revised Maxwell model in \cite{Sader} for compressible viscoelastic fluid (isentropic case). Moreover, local well-posedness and singular limit of isentropic compressible Navier-Stokes equations with revised Maxwell's law have been established by Yong \cite{W.-A. Yong 2014}. Later, Hu and Racke \cite{Hu and Racke} study the non-isentropic regime. They proved global well-posedness for small initial data and investigated the relaxation limit as $\tau_1=\tau_2=\tau\rightarrow0$.

The question of the low Mach number limit in fluid dynamics has received considerable attention and has been the subject of research for more than 30 years. Some fundamental facts on incompressible limit in fluid dynamics have been established by Klainerman and Majda in \cite{S.Klaierman and A.Majda} for local smooth solution in the case of no physical boundary condition and with well-prepared initial data. Hoff, Danchin, Lions and Masmoudi \cite{Hoff, R. Danchin, P.-L. Lions.98, N. Masmoudi.99} also made a great contribution to low Mach number limit for the isentropic compressible Navier-Stokes equations under different conditions. For the low Mach number limit of the full compressible Navier-Stokes equations, see \cite{Alazard} for local classical solution with the effects of large temperature variations and thermal condition, \cite{Guo} for local strong solution in bounded region with Dirichlet boundary condition for velocity and Neumann boundary condition for temperature,
and \cite{Ou3, S. Jiang, C. Dou} in one-dimensional and three-dimensional regime. Also see \cite{E. Feireisl,Feireisl,F.Li,Y. Li} for more results. The whole work above were aimed at the classical compressible Navier-Stokes equations with different initial-boundary value in bounded and (or) unbounded regions.

Motivated by \cite{Shuxing Zhang} where the author investigated the low Mach number limit of the full compressible Navier-Stokes equations with Cattaneo's heat transfer law and rigorously justified that, in the framework of classical solutions with small density, temperature and heat flux variations, the solutions of the Navier-Stokes-Cattaneo system converge to that of the incompressible Navier-Stokes equations as the Mach number tends to zero. The main purpose of this paper is to present the rigorous justification of the low Mach number limit of the Navier-Stokes equations with revised Maxwell's relation for small variations of density and temperature. We shall focus our study on the polytropic ideal gas obeying the perfect relations:
 \begin{equation}\label{polytropic ideal gas}
   p=\Re\rho\theta, \quad  e=C_V\theta,
\end{equation}
 where $\Re$ and $C_V>0$ are the generic gases constant and the specific heat at constant volume, respectively. The ratio of specific heats is $\gamma=1+\frac{\Re}{C_V}$. Then we can take $\Re=1$ without loss of generality.

We rewrite the system $\eqref{c N-S},\eqref{Fourier},\eqref{revised Maxwell}$ in the following form
\begin{equation}\label{N-S-C-M}
  \begin{cases}
    \partial_t\rho+u\nabla\rho+\rho \text{div}u=0,\\
    \rho\partial_t u+\rho u\nabla u+\nabla(\rho\theta)=\text{div}S_1+\nabla S_2,\\
    \rho e_\theta\partial_t\theta+\rho e_\theta u\nabla\theta+\theta p_\theta \text{div}u-\kappa\triangle\theta=(S_1+S_2I_3)\nabla u,\\
    \tau_1\partial_tS_1+S_1=\mu\left(\nabla u+(\nabla u)^\top -\frac{2}{3}\text{div}uI_3\right),\\
    \tau_2\partial_tS_2+S_2=\lambda \text{div}u.
  \end{cases}
\end{equation}

Scale $\rho,u,\theta,S_1,S_2$ in the following way:
 \begin{equation*}
 \begin{aligned}
 &\rho(x,t)=\widetilde{\rho}(x,\epsilon t),\quad u(x,t)=\epsilon \widetilde{u}(x,\epsilon t),\quad \theta(x,t)=\widetilde{\theta}(x,\epsilon t),\\
 & S_1(x,t)=\epsilon^2 \widetilde{S_1}(x,\epsilon t),\quad S_2(x,t)=\epsilon^2 \widetilde{S_2}(x,\epsilon t),\\
 \end{aligned}
\end{equation*}
and assume that the coefficients  $\mu,\lambda,\kappa,\tau_1$ and $\tau_2$ are small and scaled as:
 $$
 \mu=\epsilon\mu^\epsilon,\quad \lambda=\epsilon\lambda^\epsilon, \quad  \kappa=\epsilon\kappa^\epsilon,
    \quad \tau_1=\frac{\tau_1^\epsilon}{\epsilon},\quad \tau_2=\frac{\tau_2^\epsilon}{\epsilon},
 $$
 where $\epsilon>0\in(0,1)$ is the reference Mach number and normalized coefficients
 $\mu^\epsilon, \lambda^\epsilon,\kappa^\epsilon,\tau_1^\epsilon$ and $\tau_2^\epsilon$
 satisfy
 $$
 \mu^\epsilon\rightarrow\bar{\mu}, \quad \lambda^\epsilon\rightarrow\bar{\lambda},  \quad \kappa^\epsilon\rightarrow\bar{\kappa}, \quad \tau_1^\epsilon\rightarrow0, \quad \tau_2^\epsilon\rightarrow0,\quad \text{as} \quad \epsilon\rightarrow 0.
 $$

 We recover the non-dimensional system as follows (still denoted by $\rho,u,\theta,S_1,S_2$):
 \begin{equation}\label{n-d-N-S-C-M}
  \begin{cases}
    \partial_t\rho+u\nabla\rho+\rho \text{div}u=0,\\
    \rho\partial_t u+\rho u\nabla u+\frac{\nabla(\rho\theta)}{\epsilon^2}=\text{div}S_1+\nabla S_2,\\
    \rho e_\theta\partial_t\theta+\rho e_\theta u\nabla\theta+\theta p_\theta \text{div}u-\kappa\triangle\theta=\epsilon^2(S_1+S_2I_3)\nabla u,\\
    \tau_1^\epsilon\partial_tS_1+S_1=\mu^\epsilon\left(\nabla u+(\nabla u)^\top -\frac{2}{3}\text{div}uI_3\right),\\
    \tau_2^\epsilon\partial_tS_2+S_2=\lambda^\epsilon \text{div}u.
  \end{cases}
\end{equation}
  We further restrict ourselves to the small density and temperature variations, so we assume that
 $$
 \rho=1+\epsilon\eta, \quad \theta=1+\epsilon\phi.
 $$
Then we can rewrite \eqref{n-d-N-S-C-M} (for $ \eta,u,\phi,S_1,S_2 $) as
\begin{equation}\label{non-dimensional system}
  \begin{cases}
    \partial_t\eta+u\nabla\eta+\frac{1+\epsilon\eta}{\epsilon}\text{div}u=0,\\
    (1+\epsilon\eta)(\partial_tu+u\cdot\nabla u)+\frac{1}{\epsilon}((1+\epsilon\phi)\nabla\eta+(1+\epsilon\eta)\nabla\phi)=\text{div}S_1+\nabla S_2,\\
    (1+\epsilon\eta)(\partial_t\phi+u\cdot\nabla\phi)+\frac{\gamma-1}{\epsilon}(1+\epsilon\phi)(1+\epsilon\eta)\text{div}u
            -\kappa^\epsilon\triangle\phi=\epsilon(S_1+S_2I_3)\nabla u,\\
    \tau_1^\epsilon\partial_tS_1+S_1=\mu^\epsilon\left(\nabla u+(\nabla u)^\top -\frac{2}{3}\text{div}uI_3\right),\\
    \tau_2^\epsilon\partial_tS_2+S_2=\lambda^\epsilon \text{div}u.
  \end{cases}
\end{equation}
The system \eqref{non-dimensional system} is equipped with the initial data
\begin{equation}\label{compressible initial data}
  (\eta,u,\phi,S_1,S_2)|_{t=0}=(\eta_0(x),u_0(x),\phi_0(x),S_{10}(x),S_{20}(x)).
\end{equation}
 The formal limit as $\epsilon\longrightarrow 0$ of \eqref{non-dimensional system} is the incompressible Navier-Stokes equations (we suppose that the limit  $u\longrightarrow w $ exists.)
 \begin{equation}\label{incompressible n-s}
   \begin{cases}
     \partial_tw+(w\cdot\nabla)w+\nabla\pi=\bar{\mu}\triangle w,\\
     \text{div}w=0,
    \end{cases}
 \end{equation}
with the initial data
\begin{equation}\label{incompressible initial data}
  w|_{t=0}=w_0(x),
\end{equation}
where $\pi$ represents the limiting pressure.

In the present paper, we shall establish the above limit rigorously in a convex compact subset of G where G denotes the physical state spaces of $(\eta,u,\phi,S_1,S_2)$. Moreover, using the theory of symmetric hyperbolic parabolic system, the system \eqref{non-dimensional system} admits a smooth solution on the time interval where the smooth solution of the incompressible Navier-Stokes equations exists. The main result of this paper is stated here.
\begin{theorem}\label{th1.1}
Let $s\geq4$ be integers. Suppose that the initial data \eqref{compressible initial data} satisfy
$$
\left|\left|\left(\eta_0(x),u_0(x)-w_0(x),\phi_0(x)\right)\right|\right|_s=O(\epsilon),
$$
$$
\left|\left|\left(S_{10}(x)-\mu^\epsilon(\nabla w_0(x)+(\nabla w_0(x))^\top),S_{20}(x)\right)\right|\right|_s=O(\sqrt{\epsilon}).
$$
Let $(w, \pi)$ be a smooth solution to the system $\eqref{incompressible n-s}\sim\eqref{incompressible initial data}$ on $[0,T_*]\times G$. We assume $\pi\in C([0, T_*], H^{s+2})\cap C^1([0, T_*], H^{s+1})$ and $w\in C([0, T_*], H^{s+3})\cap C^1([0, T_*], H^{s+1})$.
Then then there exists a constant $\epsilon_0>0$, such that, for all $\epsilon<\epsilon_0$, the system \eqref{non-dimensional system} with initial data \eqref{compressible initial data} has a unique smooth solution $(\rho,u,\phi,S_{1},S_{2})\in C([0, T_*], H^{s})$. Moreover,there exists a positive constant $K>0$,independent of $\epsilon$, such that, for all $\epsilon<\epsilon_0$,
$$
\sup_{t\in[0,T_*]}\left|\left|(\eta-\frac{\epsilon\pi}{2},u-w,\phi-\frac{\epsilon\pi}{2})\right|\right|_s\leq K\epsilon,
$$
$$
\sup_{t\in[0,T_*]}\left|\left|(S_1-\mu^\epsilon(\nabla w+(\nabla w)^\top),S_2-\epsilon\lambda^\epsilon\pi)\right|\right|_s\leq K\sqrt{\epsilon}.
$$
\end{theorem}
\begin{remark}
It is needed and reasonable to require the high regularity of $(w, \pi)$ in order to estimate some terms in error system. See more details in Section 3.
\end{remark}
The rest of present paper is arranged as follows. In Section 2, we show a local existence theorem by transforming the system \eqref{non-dimensional system} into a symmetric hyperbolic-parabolic type system. Then, we construct the approximate system to the system \eqref{non-dimensional system}. In Section 3, we first establish the uniform estimations of the error system. Then we prove the convergence of relaxed compressible system \eqref{non-dimensional system} to the incompressible system \eqref{incompressible n-s} and drive the convergence rates.

Finally, let's end up this section with some notations used throughout the current paper.  We use the letter C to denote various positive constants independent of $\epsilon$.
 Denote by $\nabla^\alpha = \partial_{x_1}^{\alpha_1}\partial_{x_2}^{\alpha_2}\partial_{x_3}^{\alpha_3}$
the partial derivative of order $|\alpha|=\alpha_1 +\alpha_3 +\alpha_3$ with the multi-index $\alpha=(\alpha_1,\alpha_2,\alpha_3)$.
$L^2$ denotes the spaces of measurable function which are square integrable, with the norm $||\cdot||$. $H^s(s\geq0)$ denotes the Sobolev space of $L^2$-functions $f$ whose derivatives $\partial_x^jf,j=1,\cdots s$ are also $L^2$-functions, with the norm $||\cdot||_s$.
Let T and B be a positive constant and a Banach space, respectively. $C^k(0,T;B)(k\geq0)$ denotes the space of B-valued k-times continuously differentiable functions on $[0,T]$, and $L^2(0,T;B)$ denotes the space of B-valued $L^2$-functions on $[0,T]$. The corresponding space B-valued functions on  $[0,\infty]$ are defined similarly.
\section{\textbf{Local existence and construction of approximation solutions}}
We first show the local existence of the solutions of \eqref{non-dimensional system}. To this end, we rewrite \eqref{non-dimensional system} in the form:
\begin{equation}\label{2.1}
   \begin{cases}
\partial_t\eta+u\nabla\eta+\frac{1+\varepsilon\eta}{\epsilon}\text{div}u=0,\\
(1+\epsilon\eta)(\partial_tu+u\cdot\nabla u)+\frac{1}{\epsilon}((1+\epsilon\phi)\nabla\eta+(1+\epsilon\eta)\nabla\phi)-(\text{div}S_1+\nabla S_2)=0,\\
(1+\epsilon\eta)(\partial_t\phi+u\cdot\nabla\phi)+\frac{\gamma-1}{\epsilon}(1+\epsilon\phi)(1+\epsilon\eta)\text{div}u-\kappa^\epsilon\triangle\phi
          =\epsilon[(S_1+S_2I_3)\nabla u],\\
\tau_1^\epsilon\partial_tS_1+S_1-\mu^\epsilon\left(\nabla u+(\nabla u)^\top -\frac{2}{3}\text{div}uI_3\right)=0,\\
\tau_2^\epsilon\partial_tS_2+S_2-\lambda^\epsilon \text{div}u=0.
   \end{cases}
 \end{equation}
 Without loss of generality, we assume $S_1$ take the following form:
 $$S_1=\begin{bmatrix}
a_{11}&a_{12}&a_{13}\\
a_{12}&a_{22}&a_{23}\\
a_{13}&a_{23}&-(a_{11}+a_{22})
\end{bmatrix}.$$
 Let $U^\epsilon=(\eta,u,a_{11},a_{12},a_{13},a_{22},a_{23},S_2)$. Then, we have
\begin{equation}\label{2.2}
  \begin{cases}
  A_0\partial_tU^\varepsilon+\sum_{j=1}^{3}A_j\partial_jU^\varepsilon+BU^\varepsilon=F,\\
  (1+\epsilon\eta)\partial_t\phi-\kappa^\epsilon\triangle\phi=H.
\end{cases}
\end{equation}
Here
$
A_0=diag\left\{\frac{1+\epsilon\phi}{1+\epsilon\eta},(1+\epsilon\eta),(1+\epsilon\eta),(1+\epsilon\eta),\frac{3\tau_1^\epsilon}{4\mu^\epsilon},
\frac{\tau_1^\epsilon}{\mu^\epsilon},\frac{\tau_1^\epsilon}{\mu^\epsilon},\frac{3\tau_1^\epsilon}{4\mu^\epsilon},\frac{\tau_1^\epsilon}{\mu^\epsilon},
\frac{\tau_2^\epsilon}{\lambda^\epsilon}\right\},
$\\
$ B=diag\left\{0,0,0,0,\frac{3}{4\mu^\epsilon},
\frac{1}{\mu^\epsilon},\frac{1}{\mu^\epsilon},\frac{3}{4\mu^\epsilon},\frac{1}{\mu^\epsilon},\frac{1}{\lambda^\epsilon}\right\},
F=\left(0,\frac{(1+\epsilon\eta)}{\epsilon}\nabla\phi,0,0,0,0,0,0\right),\\
     H=\epsilon[(S_1+S_2I_3)\nabla u]-\frac{\gamma-1}{\epsilon}(1+\epsilon\phi)(1+\epsilon\eta)divu-(1+\epsilon\eta)u\cdot\nabla\phi,
$
$$
\sum_{j=1}^3A_j\xi_j=\begin{bmatrix}
  \frac{1+\epsilon\phi}{1+\epsilon\eta}u\xi & \frac{1}{\epsilon}(1+\epsilon\phi)\xi & 0_{1\times5} & 0\\

  \frac{1}{\epsilon}(1+\epsilon\phi)\xi^\top & (1+\epsilon\eta)u\xi I_3 & C_{3\times5} & -\xi^\top \\

  0_{5\times1} & D_{5\times3} & 0_{5\times5} & 0_{5\times1}  \\

  0 & -\xi & 0_{1\times5} & 0 \\
\end{bmatrix},
$$
where
$$
C_{3\times5}=\begin{bmatrix}
-\xi_1 & -\xi_2 & -\xi_3 & 0 & 0 \\
0 & -\xi_1 & 0 & -\xi_2 & -\xi_3 \\
\xi_3 & 0 & -\xi_1 & \xi_3 & -\xi_2 \\
\end{bmatrix},\quad
D_{5\times3}=\begin{bmatrix}
-\xi_1 & \frac{1}{2}\xi_2 & \frac{1}{2}\xi_3 \\
-\xi_2 & -\xi_1 & 0 \\
-\xi_3 & 0 & -\xi_1 \\
\frac{1}{2}\xi_1 & -\xi_2 & \frac{1}{2}\xi_3 \\
0 & -\xi_3 & -\xi_2 \\
\end{bmatrix},
$$
for each $\xi\in\mathcal{S}^3$.

We do the transformation: $ b_{11}=\frac{a_{11}+a_{22}}{2},b_{22}=\frac{a_{11}-a_{22}}{2}$ which particularly implies
 $ a_{11}=b_{11}+a_{22},a_{22}=b_{11}-b_{22}$ and let $\widetilde{U^\epsilon}=(\eta,u,b_{11},a_{12},a_{13},b_{22},a_{23},S_2)$.
  Then, it is easily to show that the system \eqref{2.2} is a symmetric hyperbolic parabolic system (see \cite{Hu and Racke} for more details). Therefore, we obtain the local existence of solutions for the system \eqref{2.1} .
\begin{lemma}[Local existence]\label{le2.1}
 Under the conditions in Theorem \ref{th1.1}, the system \eqref{2.1} alow a unique classical solutions $(\rho,u,\phi,S_{1},S_{2})$ for each sufficiently small $\varepsilon$ on $[0,T^\varepsilon]\times G$ satisfying
 $$(\rho,u,\phi,S_{1},S_{2})\in C([0,T^\varepsilon]\times H^s)\cap C^1([0,T^\varepsilon]\times H^{s-1}).$$
\end{lemma}
Now we define
$$
T_\epsilon=sup\{T^\epsilon:(\rho,u,\phi,S_{1},S_{2})\in C([0,T^\epsilon]\times H^s)\}.
$$
We shall make use of the following convergence-stability lemma to show that $\liminf_{\epsilon\rightarrow0} T_\varepsilon>0.$

\begin{lemma}$($\cite{Wen-An Yong.1999,Shuxing Zhang}$)$\label{le2.2}
 Let $s\geq4$. Suppose that $U^\epsilon_0 \in G_0$ and the following convergence assumption (H) holds.

 (H) For each $\epsilon$, there exists $T_*\geq0$ and $U_\epsilon\in L^\infty([0,T_*];H^s)$ satisfying
$$ \bigcap_{x,t,\epsilon}U_\epsilon\{(x,t)\}\subset\subset G$$
such that, for $t\in(0,\min\{T_*,T_\epsilon\})$,
$$
\sup_{x,t}|U^\epsilon-U_\epsilon|=o(1), \quad \sup_t||U^\epsilon-U_\epsilon||_s=O(1).
$$
Then there exists an $\bar\epsilon > 0$ such that,  $\forall\epsilon \in(0,\bar\epsilon]$, it holds that $T_\epsilon> T_*$.
\end{lemma}
Now, we construct the approximation $U_\varepsilon=(\eta_\epsilon,u_\epsilon,\phi_\epsilon,S_{1\epsilon},S_{2\epsilon})$ to the original system \eqref{2.1} with
$ \eta_\epsilon=\frac{\epsilon\pi}{2}, u_\epsilon=w, \phi_\epsilon=\frac{\epsilon\pi}{2}, S_{1\epsilon}=\mu^\epsilon(\nabla w+(\nabla w)^\top), S_{2\epsilon}=\epsilon\lambda^\epsilon\pi,$
where $(w,\pi)$ is the smooth solution to the system \eqref{incompressible n-s},\eqref{incompressible initial data}.  It is easy to verify that $(\eta_\epsilon,u_\epsilon,\phi_\epsilon,S_{1\epsilon},S_{2\epsilon})$ satisfy
\begin{equation}\label{2.3}
   \begin{cases}
\partial_t\eta_\epsilon+u_\epsilon\nabla\eta_\epsilon+\frac{1+\varepsilon\eta_\epsilon}{\epsilon}\text{div}u_\epsilon
         =\frac{\epsilon}{2}(\pi_t+w\nabla\pi)=:f_1,\\  %% 1
(1+\epsilon\eta_\epsilon)(\partial_tu_\epsilon+u_\epsilon\cdot\nabla u_\epsilon)+\frac{1}{\epsilon}((1+\epsilon\phi_\epsilon)\nabla\eta_\epsilon
           +(1+\epsilon\eta_\epsilon)\nabla\phi_\epsilon)-(\text{div}S_{1\epsilon}+\nabla S_{2\epsilon})\\
     \quad =\frac{\epsilon^2\pi}{2}(w_t+w\cdot\nabla w+\nabla\pi)-\epsilon\lambda^\epsilon\nabla\pi=:f_2,\\  %% 2
(1+\epsilon\eta_\epsilon)(\partial_t\phi_\epsilon+u_\epsilon\cdot\nabla\phi_\epsilon)
           +\frac{\gamma-1}{\epsilon}(1+\epsilon\phi_\epsilon)(1+\epsilon\eta_\epsilon)\text{div}u_\epsilon-\kappa^\epsilon\triangle\phi_\epsilon\\
     \quad =(\frac{\epsilon}{2}+\frac{\epsilon^3\pi}{4})(\pi_t+w\nabla\pi)-\frac{\epsilon\kappa}{2}\triangle\pi=:f_3,  \\   %% 3
\tau_1^\epsilon\partial_tS_{1\epsilon}+S_{1\epsilon}-\mu^\epsilon
           \left(\nabla u_\epsilon+(\nabla u_\epsilon)^\top-\frac{2}{3}\text{div}u_\epsilon I_3\right)
     =\tau_1^\epsilon\mu^\epsilon(\nabla w_t+(\nabla w_t)^\top)=:f_4,\\   %% 4
\tau_2^\epsilon\partial_tS_{2\epsilon}+S_{2\epsilon}-\lambda^\epsilon \text{div}u_\epsilon=\epsilon\lambda^\epsilon(\tau_2\pi_t+\pi)=:f_5.\\ %% 5
   \end{cases}
 \end{equation}
Due to the regularity assumptions on $(w,\pi)$ in Theorem \ref{th1.1} and $\tau_1^\epsilon=O(\epsilon)$ and $\tau_2^\epsilon=O(\epsilon)$, we have the fact:
\begin{equation}\label{f}
  ||f_i||_s\leq C\epsilon, \quad i=1,2,3,4,5.
\end{equation}
\section{\textbf{The uniform estimates of the error system}}

In this section, we first establish the uniform estimates of the error system. Then based on this result, we verify the convergence assumption (H) and accomplish the proof of Theorem \ref{th1.1}. To this end, we now introduce
$$
\eta^d=\eta-\eta_\epsilon, u^d=u-u_\epsilon, \phi^d=\phi-\phi_\epsilon, S_1^d=S_1-S_{1\epsilon}, S_2^d=S_2-S_{2\epsilon}
$$
and utilize \eqref{2.1} and \eqref{2.3}. Then, we get the error system in the following form
\begin{equation}\label{3.1}
  \begin{cases}
\partial_t\eta^d+ (u\cdot\nabla)\eta^d+\frac{1+\epsilon\eta}{\epsilon}\text{div}u^d =-u^d\nabla\eta_\epsilon-f_1, \\  %%% 第一个方程
\partial_tu^d+(u\cdot\nabla)u^d+\frac{1}{\epsilon}\left[\frac{(1+\epsilon\phi)\nabla\eta^d}{(1+\epsilon\eta)}+\nabla\phi^d\right]
               -\frac{1}{(1+\epsilon\eta)}(\text{div}S_1^d+\nabla S_2^d)\\
    \qquad \qquad =-\frac{\epsilon\eta^d\partial_tu_\epsilon}{(1+\epsilon\eta)}-\frac{((1+\epsilon\eta) u-
               (1+\epsilon\eta_\epsilon) u_\epsilon)}{(1+\epsilon\eta)}\cdot\nabla u_\epsilon
               -\frac{(\eta^d\nabla\phi_\epsilon+\phi^d\nabla\eta_\epsilon)}{(1+\epsilon\eta)} -\frac{1}{(1+\epsilon\eta)}f_2,\\   %%% 第二个方程
\partial_t\phi^d+(u\cdot\nabla)\phi^d+\frac{\gamma-1}{\epsilon}(1+\epsilon\phi)\text{div}u^d-\frac{\kappa^\epsilon}{(1+\epsilon\eta)}\triangle\phi^d\\
   \qquad \qquad =-\frac{\epsilon\eta^d\partial_t\phi_\epsilon}{(1+\epsilon\eta)}-\frac{((1+\epsilon\eta) u-(1+\epsilon\eta_\epsilon) u_\epsilon)}
       {(1+\epsilon\eta)}\cdot\nabla \phi_\epsilon+\frac{\epsilon[(S_1+S_2I_3)\nabla u]}{(1+\epsilon\eta)}-\frac{1}{(1+\epsilon\eta)}f_3,\\
\tau_1^\epsilon\partial_tS_1^d+ S_1^d-\mu^\epsilon\left(\nabla u^d+(\nabla u^d)^\top -\frac{2}{3}\text{div}u^dI_3\right)=-f_4,\\  %%% 第四个方程
\tau_2^\epsilon\partial_tS_2^d + S_2^d-\lambda^\epsilon \text{div}u^d=-f_5.\\ %%% 第五个方程
  \end{cases}
\end{equation}
Define
$$
E=E(t):=\left|\left|(\rho^d,u^d,\phi^d,\sqrt{\tau_1^\epsilon}S_1^d,\sqrt{\tau_2^\epsilon}S_2^d)\right|\right|_s.
$$
Note that
$$
||(\rho_\epsilon,u_\epsilon,\phi_\epsilon,S_{1\epsilon},S_{2\epsilon})||_s \leq C
$$
and
$$
||(\rho,u,\phi)||_s \leq C(1+E), \quad  ||S_1||_s\leq C\left(1+\frac{1}{\sqrt{\tau_1^\epsilon}}E\right),
 \quad  ||S_2||_s\leq C\left(1+\frac{1}{\sqrt{\tau_2^\epsilon}}E\right).
$$

Our aim is to show that, for small $\epsilon$ and for all $t\in(0,\min\{T_\epsilon, T_*\})$, $E(t)\leq C\epsilon$, provided that $E(0)\leq C\epsilon_0$ is suitably small. Now, we list the Moser-type inequalities (see \cite{ Racke}) and the nonlinear Gronwall-type inequality (see \cite{Wen-An Yong.1999}) which will be used in the following lemma \ref{le3.2}.
\begin{lemma}[Moser-type inequalities]\label{le3.1}
 Let $k\in\mathbb{N}$. Then there is a constant $c=c(k,n)>0$ such that for all f,g $\in W^{k,2}\cap L^\infty$,and $\alpha\in \mathbb{N}_0^n,|\alpha|\leq k$, the following inequalities hold:
 \begin{align*}
  ||\nabla^\alpha(fg)||\leq c(||f||_{\infty} \cdot||\nabla^kg||+||g||_{\infty}\cdot||\nabla^kf||), \\
  ||\nabla^\alpha(fg)-f\nabla^\alpha g||\leq c(||\nabla f||_{\infty} \cdot||\nabla^{k-1}g||+||g||_{\infty}\cdot||\nabla^kf||).
 \end{align*}
\end{lemma}

\begin{lemma}[Nonlinear Gronwall-type inequality]\label{le3.3}
Suppose $\sigma(t)$ is a positive $C^1$-function of $t\in[0,T)$ with $T\leq\infty$, $m>1$ and $b_1(t),b_2(t)$ are integrable on $[0,T)$. If
$$
\sigma'(t)\leq b_2\sigma^m(t)+b_1\sigma(t),
$$
then there exists a $\delta>0$, depending only on m, $C_{1b}$ and $C_{2b}$, such that
$$
\sup_{t\in[0,T)}\sigma(t)\leq e^{C_{1b}},
$$
whenever $\sigma(0)\in(0,\delta]$. Here
$$
C_{1b}=\sup_{t\in[0,T)}\int_0^tb_1(s)ds  \quad
 and \quad
 C_{2b}=\sup_{t\in[0,T)}\int_0^tmax\{b_2(s),0\}ds.
$$
\end{lemma}

The following lemma is essential to prove the Theorem \ref{th1.1}.
\begin{lemma}\label{le3.2}
 For suitably small $\epsilon\in(0,1)$, we have
 \begin{equation*}
   \frac{dE^2}{dt}\leq C(1+E^2)E^2+C\epsilon^2.
 \end{equation*}
\end{lemma}
\begin{proof}
Applying $\nabla^\alpha(|\alpha|\leq s)$ to the error system \eqref{3.1}, multiplying the result by
 $
 \frac{1+\epsilon\phi}{1+\epsilon\eta}\nabla^\alpha \eta^d,\\
   (1+\epsilon\eta)\nabla^\alpha u^d,  \frac{1+\epsilon\eta}{(\gamma-1)(1+\epsilon\phi)}\nabla^\alpha \phi^d,
 \frac{1}{2\mu^\epsilon}\nabla^\alpha S_1^d,\frac{1}{\lambda^\epsilon}\nabla^\alpha S_2^d
 $
and integrating with respect to $x$, we get
\begin{equation}\label{3.2}
   \begin{aligned}
      \frac{1}{2}\frac{d}{dt}\int &\left\{\frac{1+\epsilon\phi}{1+\epsilon\eta}(\nabla^\alpha\eta^d)^2+ (1+\epsilon\eta)(\nabla^\alpha u^d)^2+ \frac{1+\epsilon\eta}{(\gamma-1)(1+\epsilon\phi)}(\nabla^\alpha \phi^d)^2\right.\\
       &\left.+ \frac{\tau_1^\epsilon}{2\mu^\epsilon}(\nabla^\alpha S_1^d)^2
              + \frac{\tau_2^\epsilon}{\lambda^\epsilon}(\nabla^\alpha S_2^d)^2\right\}dx\\
        +  \int &\left\{\frac{1}{2\mu^\epsilon}(\nabla^\alpha S_1^d)^2
                      + \frac{\kappa^\epsilon}{(\gamma-1)(1+\epsilon\phi)}(\nabla^{\alpha+1} \phi^d)^2
                      + \frac{1}{\lambda^\epsilon}(\nabla^\alpha S_2^d)^2\right\}dx\\
      \leq&\sum_{i=1}^3 | T_i |+\sum_{i=1}^5 | F_i |+\sum_{i=1}^6 | G_i |+\sum_{i=1}^7 | N_i |+|D|,\\
   \end{aligned}
\end{equation}
 where
\begin{align*}
  & T_1=\int\frac{1}{2}\left(\frac{1+\epsilon\phi}{1+\epsilon\eta}\right)_t(\nabla^\alpha\eta^d)^2dx,
    T_2=\int\frac{1}{2}(1+\epsilon\eta)_t(\nabla^\alpha u^d)^2dx,\\
  & T_3=\int\frac{1}{2}\left(\frac{1+\epsilon\eta}{(\gamma-1)(1+\epsilon\phi)}\right)_t(\nabla^\alpha \phi^d)^2dx,
    F_1=\int\frac{1+\epsilon\phi}{1+\epsilon\eta}\nabla^\alpha\eta^d \nabla^\alpha f_1 dx,\\
  & F_2=\int(1+\epsilon\eta)\nabla^\alpha u^d\nabla^\alpha\left(\frac{f_2}{1+\epsilon\eta}\right)dx,
    F_3=\int\frac{1+\epsilon\eta}{(\gamma-1)(1+\epsilon\phi)}\nabla^\alpha \phi^d\nabla^\alpha\left(\frac{f_3}{1+\epsilon\eta}\right)dx,\\
  & F_4=\int\frac{1}{2\mu^\epsilon}\nabla^\alpha S_1^d\nabla^\alpha f_4dx,
    F_5=\int\frac{1}{\lambda^\epsilon}\nabla^\alpha S_2^d\nabla^\alpha f_5dx,\\
  & G_1=\int\frac{1+\epsilon\phi}{1+\epsilon\eta}\nabla^\alpha\eta^d\nabla^\alpha(u\cdot\nabla\eta^d)dx,
    G_2=\int(1+\epsilon\eta)\nabla^\alpha u^d\nabla^\alpha\left(u\cdot\nabla u^d\right)dx,\\
  & G_3=\int\frac{1+\epsilon\eta}{(\gamma-1)(1+\epsilon\phi)}\nabla^\alpha\phi^d\nabla^\alpha\left(u\cdot\nabla \phi^d\right)dx,\\
  & G_4=\frac{1}{\epsilon}\int\frac{1+\epsilon\phi}{1+\epsilon\eta}\nabla^\alpha\eta^d\nabla^\alpha
             \left[(1+\epsilon\eta)\text{div}u^d\right]+(1+\epsilon\eta)\nabla^\alpha u^d\nabla^\alpha
             \left[\frac{(1+\epsilon\phi)}{(1+\epsilon\eta)}\nabla\eta^d\right]dx,\\
  & G_5=\frac{1}{\epsilon}\int(1+\epsilon\eta)\nabla^\alpha u^d\nabla^{\alpha+1}\phi^d
            +\frac{1+\epsilon\eta}{1+\epsilon\phi}\nabla^\alpha\phi^d\nabla^\alpha[(1+\epsilon\phi)\text{div}u^d] dx,\\
  & G_6=\int(1+\epsilon\eta)\nabla^\alpha u^d\nabla^\alpha \left(\frac{1}{(1+\epsilon\eta)}(\text{div}S_1^d+\nabla S_2^d)\right)\\
  &  \qquad  \quad  + \frac{1}{2}\nabla^\alpha S_1^d\nabla^\alpha
             \left(\nabla u^d+(\nabla u^d)^\top -\frac{2}{3}\text{div}u^dI_3\right) + \nabla^\alpha S_2^d\nabla^\alpha(\text{div}u^d)dx,\\
  & N_1=\int\frac{1+\epsilon\phi}{1+\epsilon\eta}\nabla^\alpha\eta^d\nabla^\alpha(u^d\nabla\eta_\epsilon)dx,
    N_2=\int(1+\epsilon\eta)\nabla^\alpha u^d\nabla^\alpha\left(\frac{\epsilon\eta^d\partial_tu_\epsilon}{1+\epsilon\eta}\right)dx,\\
  & N_3=\int(1+\epsilon\eta)\nabla^\alpha u^d\nabla^\alpha\left[\frac{(1+\epsilon\eta)u
             -(1+\epsilon\eta_\epsilon) u_\epsilon}{(1+\epsilon\eta)}\cdot\nabla u_\epsilon\right]dx,\\
  & N_4=\int(1+\epsilon\eta)\nabla^\alpha u^d\nabla^\alpha
             \left(\frac{\eta^d\nabla\phi_\epsilon+\phi^d\nabla\eta_\epsilon}{1+\epsilon\eta}\right)dx,
    N_5=\int\frac{(1+\epsilon\eta)\nabla^\alpha\phi^d}{(\gamma-1)(1+\epsilon\phi)}\nabla^\alpha
             \left(\frac{\epsilon\eta^d\partial_t\phi_\epsilon}{1+\epsilon\eta}\right)dx,\\
  & N_6=\int\frac{1+\epsilon\eta}{(\gamma-1)(1+\epsilon\phi)}\nabla^\alpha\phi^d\nabla^\alpha\left[\frac{(1+\epsilon\eta) u-
             (1+\epsilon\eta_\epsilon) u_\epsilon}{(1+\epsilon\eta)}\cdot\nabla \phi_\epsilon\right]dx,\\
  & N_7=\int\frac{1+\epsilon\eta}{(\gamma-1)(1+\epsilon\phi)}\nabla^\alpha\phi^d\nabla^\alpha
             \left[\frac{\epsilon(S_1+S_2I_3)\nabla u}{(1+\epsilon\eta)}\right]dx.\\
  &D=\int\frac{1+\epsilon\eta}{(\gamma-1)(1+\epsilon\phi)}\nabla^\alpha\phi^d\left[\nabla^\alpha\left(\frac{\kappa^\epsilon}{(1+\epsilon\eta)}
              \triangle \phi^d\right)-\frac{\kappa^\epsilon}{(1+\epsilon\eta)}\nabla^\alpha\triangle \phi^d\right]\\
  &  \qquad  \quad + \nabla\left[\frac{\kappa^\epsilon}{(\gamma-1)(1+\epsilon\phi)}\right]\nabla^\alpha\phi^d\nabla^{\alpha+1}\phi^d dx.
\end{align*}
Note that in the time interval $[0, min({T_*, T_\epsilon})]$, both $(\eta,u,\phi,S_1,S_2)(t)$ and $(\eta_\epsilon,u_\epsilon,\phi_\epsilon,S_{1\epsilon},S_{2\epsilon})(t)$ are regular enough and take values in a convex compact subset of G. It's easy to know
\begin{align*}
& |(F_1,F_2,F_3)|\leq C\epsilon^2+C||\nabla^\alpha(\eta^d,u^d,\phi^d)||^2\leq C\epsilon^2+CE^2,\\
& |F_4|\leq C\epsilon^2+\delta_1||\nabla^\alpha S_1^d||^2, \quad
  |F_5|\leq C\epsilon^2+\delta_2||\nabla^\alpha S_2^d||^2.
\end{align*}
We now estimate $G_1\sim G_6$.
\begin{align*}
|G_1|&\leq\left|\int\frac{1+\epsilon\phi}{1+\epsilon\eta}\nabla^\alpha\eta^d\left[\nabla^\alpha(u\nabla\eta^d)-u\nabla^{\alpha+1}\eta^d\right]
            +\frac{1+\epsilon\phi}{1+\epsilon\eta}\nabla^\alpha\eta^du\nabla^{\alpha+1}\eta^d dx\right|  \\
     &\leq C(||u||_{\infty}||\nabla^\alpha\eta^d||+||\nabla\eta^d||_{\infty}||\nabla^\alpha u||)||\nabla^\alpha\eta^d||
            +\left|\left|\text{div}\left(\frac{(1+\epsilon\phi)u}{1+\epsilon\eta}\right)\right|\right|_{\infty}||\nabla^\alpha\eta^d||^2  \\
     &\leq C(1+E)E^2.
\end{align*}
$G_2\sim G_3$ can be estimated in the similar way. As for the singular terms $G_4\sim G_6$, we have
\begin{align*}
G_4=\frac{1}{\epsilon}\int & (1+\epsilon\phi)\text{div}(\nabla^\alpha u^d\nabla^\alpha\eta^d)\\
   & +\frac{1+\epsilon\phi}{1+\epsilon\eta}\nabla^\alpha\eta^d\{\nabla^\alpha[(1+\epsilon\eta)\text{div}u^d]-(1+\epsilon\eta)
        \nabla^\alpha \text{div}u^d\}\\
   & +(1+\epsilon\eta)\nabla^\alpha u^d\left[\nabla^\alpha
      \left(\frac{1+\epsilon\phi}{1+\epsilon\eta}\nabla\eta^d\right)-\frac{1+\epsilon\phi}{1+\epsilon\eta}\nabla^\alpha(\nabla\eta^d)\right]dx\\
 \leq C||( & \nabla^\alpha u^d,\nabla^\alpha\eta^d)|| \Big{(} ||\nabla\eta||_\infty
            ||\nabla^\alpha u^d||+||\text{div}u^d||_\infty||\nabla^\alpha\eta|| \\
   & +||\nabla\eta||_\infty||\nabla\phi||_\infty||\nabla^\alpha \eta^d||+||\nabla\eta^d||_\infty||\nabla^\alpha\phi||||\nabla^\alpha\eta||\Big{)}\\
 \leq C E&^2(1+E^2)
\end{align*}
and
\begin{align*}
G_5=\frac{1}{\epsilon}\int & (1+\epsilon\eta)\text{div}(\nabla^\alpha u^d\nabla^\alpha\phi^d)\\
   & +\frac{1+\epsilon\eta}{1+\epsilon\phi}\nabla^\alpha u^d\{\nabla^\alpha[(1+\epsilon\phi)\text{div}u^d]-(1+\epsilon\phi)
        \nabla^\alpha \text{div}u^d\}dx\\
 \leq C||( & \nabla^\alpha u^d,\nabla^\alpha\phi^d)|| \Big{(}||\nabla\eta||_\infty+ ||\nabla\phi||_\infty
            ||\nabla^\alpha u^d||+||\text{div}u^d||_\infty||\nabla^\alpha\phi|| \Big{)}\\
 \leq C E&^2(1+E).
\end{align*}
$G_6$ can be estimated similarly and the result is
$$
  G_6\leq\delta_3||(S_1^d,S_2^d)||_s^2+C\epsilon^2(1+E^2)E^2.
$$
Note that when estimating $G_6$, the following two essential cancelation relations are used
$$
\int\nabla^\alpha(\text{div}S_1^d)\nabla^\alpha u^ddx=-\int\frac{1}{2}\nabla^\alpha
\left(\nabla u^d+(\nabla u^d)^\top -\frac{2}{3}\text{div}u^dI_3\right)\nabla^\alpha S_1^d dx
$$
and
$$
\int\nabla^\alpha(\nabla S_2^d)\nabla^\alpha u^ddx=-\int\nabla^\alpha S_2^d\nabla^\alpha(\text{div}u^d )dx.
$$
 These two relations can be easily shown to hold by doing partial integration and using the fact that $S^d_1$ is a symmetric and traceless matrix.
As for $N_i (i=1,2,3,4,5,6) $, noting that $N_4\sim N_6$ are similar to $N_1\sim N_3$, it is enough to estimate $N_1\sim N_3$.
\begin{align*}
N_1&=\int\frac{1+\epsilon\phi}{1+\epsilon\eta}\nabla^\alpha\eta^d\nabla^\alpha(u^d\nabla\eta_\epsilon) dx  \\
   &\leq \left|\left|\frac{1+\epsilon\phi}{1+\epsilon\eta}\right|\right|_{\infty}||\nabla^\alpha\eta^d||\left(|| \nabla\eta_\epsilon||_{\infty}||\nabla^\alpha u^d||+||u^d||_{\infty}||\nabla^{\alpha+1}\eta_\epsilon||\right)\\
   &\leq C\epsilon E^2,
\end{align*}
 similarly,
\begin{align*}
N_2=\int(1+\epsilon\eta)\nabla^\alpha u^d\nabla^\alpha\left(\frac{\epsilon\eta^d\partial_tu_\epsilon}{1+\epsilon\eta}\right)dx\leq C\epsilon E^2
\end{align*}
and
\begin{align*}
N_3&=\int(1+\epsilon\eta)\nabla^\alpha u^d\nabla^\alpha\left[\frac{(1+\epsilon\eta)u
          -(1+\epsilon\eta_\epsilon) u_\epsilon}{1+\epsilon\eta}\cdot\nabla u_\epsilon\right]dx \\
   &=\int(1+\epsilon\eta)\nabla^\alpha u^d\nabla^\alpha\left[\frac{u^d+\epsilon(\eta^du+\eta_\epsilon u^d)}{1+\epsilon\eta}
         \cdot\nabla u_\epsilon\right]dx \\
   &\leq C(E^2+(1+\epsilon)E^3).
\end{align*}
We estimate $N_7$ for $1\leq|\alpha|\leq s$ and $|\alpha|=0$.\\
Case 1: for $1\leq|\alpha|\leq s$, We have
\begin{align*}
N_7&=\int\frac{1+\epsilon\eta}{(\gamma-1)(1+\epsilon\phi)}\nabla^\alpha\phi^d\nabla^\alpha
             \left[\frac{\epsilon(S_1+S_2I_3)\nabla u}{(1+\epsilon\eta)}\right]dx \\
   &=-\int\frac{1+\epsilon\eta}{(\gamma-1)(1+\epsilon\phi)}\nabla^{\alpha+1}\phi^d\nabla^{\alpha-1}
             \left[\frac{\epsilon(S_1+S_2I_3)\nabla u}{(1+\epsilon\eta)}\right]dx \\
   &=-\int\nabla\left[\frac{1+\epsilon\eta}{(\gamma-1)(1+\epsilon\phi)}\right]\nabla^\alpha\phi^d\nabla^{\alpha-1}
             \left[\frac{\epsilon(S_1+S_2I_3)\nabla u}{(1+\epsilon\eta)}\right]dx\\
   &=:N_7^1+N_7^2
\end{align*}
with
\begin{align*}
|N_7^1| & =\left|\int\frac{1+\epsilon\eta}{(\gamma-1)(1+\epsilon\phi)}\nabla^{\alpha+1}\phi^d\nabla^{\alpha-1}
             \left[\frac{\epsilon(S_1+S_2I_3)\nabla u}{(1+\epsilon\eta)}\right]dx\right|\\
        & \leq\int\delta_4(\nabla^{\alpha+1}\phi^d)^2 dx
            + \epsilon^2(1+\epsilon\eta)^2\left\{\nabla^{\alpha-1}\left[\frac{(S_1+S_2I_3)\nabla u}{(1+\epsilon\eta)}\right]\right\}^2dx \\
        & \leq \delta_4\int(\nabla^{\alpha+1}\phi^d)^2 dx+C\epsilon^2(1+E^4)+\epsilon^2E^2\left(\frac{1}{\sqrt{\tau_1^\epsilon}}+\frac{1}{\sqrt{\tau_2^\epsilon}}\right),
\end{align*}
and
\begin{align*}
|N_7^2| & =\left|\int\nabla\left[\frac{1+\epsilon\eta}{(\gamma-1)(1+\epsilon\phi)}\right]\nabla^\alpha\phi^d\nabla^{\alpha-1}
             \left[\frac{\epsilon(S_1+S_2I_3)\nabla u}{(1+\epsilon\eta)}\right]dx\right|\\
        & \leq \epsilon^2||\nabla(\eta,\phi)||_{\infty}||\nabla^\alpha\phi^d||(||(S_1,S_2,\eta)||_{\infty}||\nabla^\alpha u||+||\nabla u||_{\infty}
            ||\nabla^{\alpha-1}(S_1,S_2,\eta)||)\\
        & \leq C\epsilon^2(1+E)^2 E\left(1+E+\frac{1}{\sqrt{\tau_1^\epsilon}}E+\frac{1}{\sqrt{\tau_2^\epsilon}}E\right)\\
        & \leq C\epsilon^2(1+E^4)+ C\epsilon^2E^2(1+E^2)\left(\frac{1}{\sqrt{\tau_1^\epsilon}}+\frac{1}{\sqrt{\tau_2^\epsilon}}\right).
\end{align*}
Case 2: for $|\alpha|=0$, It holds that
\begin{align*}
N_7 &=\int\frac{\epsilon}{(\gamma-1)(1+\epsilon\phi)}\phi^d(S_1+S_2I_3)\nabla udx \\
    & \leq \epsilon||\phi^d|| ||(S_1+S_2I_3)\nabla u||\\
    & \leq C\epsilon E(1+E)\left(1+\frac{1}{\sqrt{\tau_1^\epsilon}}E+\frac{1}{\sqrt{\tau_2^\epsilon}}E\right)\\
    & \leq CE^2(1+E^2)+C\epsilon^2+\epsilon^2E^2\left(\frac{1}{\sqrt{\tau_1^\epsilon}}+\frac{1}{\sqrt{\tau_2^\epsilon}}\right).
\end{align*}
Therefore, we obtain
 $$
     N_7\leq \delta_4\int(\nabla^{\alpha+1}\phi^d)^2 dx+ CE^2(1+E^2)\left(1+\frac{\epsilon^2}{\sqrt{\tau_1^\epsilon}}+\frac{\epsilon^2}{\sqrt{\tau_2^\epsilon}}\right)+ C\epsilon^2.
 $$
While
\begin{align*}
D & =\int \frac{1+\epsilon\eta}{(\gamma-1)(1+\epsilon\phi)}\nabla^\alpha\phi^d\left[\nabla^\alpha\left(\frac{\kappa^\epsilon}{(1+\epsilon\eta)}
              \triangle \phi^d\right)-\frac{\kappa^\epsilon}{(1+\epsilon\eta)}\nabla^\alpha\triangle \phi^d\right]\\
  &  \qquad  + \nabla\left[\frac{\kappa^\epsilon}{(\gamma-1)(1+\epsilon\phi)}\right]\nabla^\alpha\phi^d\nabla^{\alpha+1}\phi^d dx.\\
  &\leq||\nabla^\alpha\phi^d||(||\triangle\phi^d||_{\infty}||\nabla^\alpha\eta||+\epsilon||\nabla(\phi,\eta)||_{\infty}||\nabla^{\alpha+1}\phi^d||)\\
  & \leq CE^2(1+E)+ C\epsilon E(1+E)||\nabla^{\alpha+1}\phi^d||\\
  & \leq CE^2(1+E)+ C\epsilon^2 E^2(1+E)^2 + \delta_5||\nabla^{\alpha+1}\phi^d||^2.
\end{align*}
We now estimate the last three terms $ T_1\sim T_3 $.
\begin{align*}
|T_i|&\leq C\epsilon||(\eta_t,\phi_t)||_\infty||\nabla^\alpha(\eta^d,u^d,\phi^d)||^2\leq C\epsilon||(\eta_t,\phi_t)||_{\infty} E^2\\
     &\leq C\epsilon(1+||(\eta^d_t,\phi^d_t)||_{\infty}) E^2\leq C(1+E^2)E^2.
\end{align*}

Based on the above estimates, we can take $\delta_1+\delta_3,\delta_2+\delta_3$ and $\delta_4+\delta_5$ suitably small such that  $(\delta_1+\delta_3)||\nabla^\alpha S_1^d||^2,(\delta_2+\delta_3)||\nabla^\alpha S_2^d||^2$
and $(\delta_4+\delta_5)||\nabla^{\alpha+1}\phi^d||^2$
can be absorbed by the integrals of
$(\nabla^\alpha S_1^d)^2,(\nabla^\alpha S_2^d)^2$ and $(\nabla^{\alpha+1}\phi^d)^2$
on the left side.
Therefore, we have the fact
$$
 \frac{dE^2}{dt}\leq C(1+E^2)E^2+C\epsilon^2.
$$
The proof is accomplished.
\end{proof}
Now we turn to prove Theorem \ref{th1.1}. We integrate the inequality in Lemma \ref{le3.2} over (0, t) with $t\leq min\{T_\epsilon, T_*\}$ to obtain
$$
E^2\leq CE(0)^2+C\int_0^t(1+E^2)E^2dt+C\int_0^t\epsilon^2 dt.
$$
With the help of the condition in Theorem \ref{th1.1} and Gronwall's lemma, we conclude that
$$
E^2\leq C\epsilon^2 \text{exp}\left\{C\int_0^t(1+E^2)dt\right\}\equiv\Phi(t).
$$
Moreover, it is easy to know
$$
\Phi'(t)\leq C\Phi(t)(1+E^2)=C\left[\Phi(t)+\Phi^2(t)\right].
$$
By employing the nonlinear Gronwall-type inequality, we conclude that there exists a constant $K$ such that
$$
E\leq K\epsilon
$$
for all $t\in(0,\min\{T_\epsilon, T_*\})$, provided that $\Phi(0)=C\epsilon^2$ is suitably small.

\end{document}